\documentclass[12pt,letterpaper]{amsart}

\advance \topmargin by -\headheight
\advance \topmargin by -\headsep

\evensidemargin \oddsidemargin
\usepackage{amssymb,amscd,pb-diagram}

\newcommand{\leg}[2]{\genfrac{(}{)}{}{}{#1}{#2}}
\newcommand\be{\begin{equation}}
\newcommand\ee{\end{equation}}
\newcommand\bea{\begin{eqnarray}}
\newcommand\eea{\end{eqnarray}}

\newcommand{\m}{{\mathfrak m}}
\newcommand{\p}{{\mathfrak p}}
\newcommand{\q}{{\mathfrak q}}

\newcommand{\oo}{{\mathcal O}}
\newcommand{\ok}{{{\mathcal O}_K}}
\newcommand{\C}{{\mathbf C}}
\newcommand{\e}{{\epsilon}}
\newcommand{\F}{{\mathbf F}}
\newcommand{\pp}{{\mathbf P}}
\newcommand{\Q}{{\mathbf Q}}
\newcommand{\R}{{\mathbf R}}
\newcommand{\Z}{{\mathbf Z}}
\newcommand{\ooh}{{\mathcal O}}

\newcommand{\vsp}{\vspace{8pt}}
\newcommand{\hsp}{\hspace{10pt}}
\newcommand{\vtsp}{\vspace{16pt}}
\newcommand{\rarr}{{\rightarrow}}
\newcommand{\onto}{{\twoheadrightarrow}}
\newcommand{\mymod}[2]{{ #1 \: (\bmod \: {#2})}}
\newcommand{\sfrac}[2]{{\textstyle\frac{#1}{#2}}}

\newtheorem{thm}{Theorem}[section]
\newtheorem{thmm}{Theorem}[section]
\newtheorem{lem}[thm]{Lemma}

\newtheorem{cor}[thm]{Corollary}
\newtheorem{defn}[thm]{Definition}

\theoremstyle{remark}
\newtheorem*{remark}{Remark}

\newtheorem*{ack}{Acknowledgment}

\newcommand\gal{{ \mbox{\rm Gal}}}
\newcommand{\trho}{{\tilde{\rho}}}
\newcommand{\mun}{{{\boldsymbol \mu}}}

\newcommand{\fq}{{{\mathbf F}_q}}
\newcommand{\fl}{{{\mathbf F}_l}}
\newcommand{\fqbar}{{\overline{\mathbf F}_q}}
\newcommand{\myline}[2]{{\langle #1, #2 \rangle}}
\newcommand{\ov}[1]{{\overline{{#1}}}}
\newcommand{\s}{{\mathcal S}}
\newcommand{\A}{{\mathcal A}}

\newcount\timehh\newcount\timemm
\timehh=\time
\divide\timehh by 60 \timemm=\time
\begin{document}
\title{Quadratic fields with cyclic $2$-class groups}

\author{Carlos Dominguez}
\address{Department of Mathematics \& Statistics, Williams College,
Williamstown, MA 01267}
\email{Carlos.R.Dominguez@williams.edu}

\author{Steven J. Miller}
\address{Department of Mathematics \& Statistics, Williams College,
Williamstown, MA 01267}
\email{Steven.J.Miller@williams.edu}

\author{Siman Wong}

\if 3\
{
  {
    \protect \protect\sc\today\ --
    \ifnum\timehh<10 0\fi\number\timehh\,:\,\ifnum\timemm<10 0\fi\number\timemm
    \protect \, \, \protect \bf DRAFT
  }
}
\fi

\address{Department of Mathematics \& Statistics, University of Massachusetts.
	Amherst, MA 01003-9305 USA}
\email{siman@math.umass.edu}





\begin{abstract}
For any integer $k\ge 1$,
 we show that there are infinitely many complex
quadratic fields whose $2$-class groups are cyclic of order
$2^k$.
The proof combines the circle method with an algebraic criterion for a
complex quadratic ideal class to be a square.\\ \ \\ In memory of David Hayes.
\end{abstract}

\thanks{The first named author was partially supported by NSF Grant DMS0850577 and Williams College. The second named author was partially supported by  NSA grant H98230-05-1-0069 and NSF Grant DMS0970067}

\maketitle


\section{Introduction}

The genus theory of Gauss completely determines the elementary $2$-subgroup of
the class group of a complex quadratic field.   In particular, we can construct
complex quadratic class groups with prescribed elementary $2$-subgroups.
Using class field theory, R\'edei \cite{redei} gave the first algorithm for
determining the complete structure of the Sylow $2$-subgroup of a quadratic
class group.  Now, genus theory readily yields an explicit criterion for a
divisor class in a quadratic class group to be a square, cf.~section
\ref{sec:cyclic} below.
This leads to a new, simplified recursive algorithm
(Shanks \cite{shanks} used the language of quadratic forms; others
(\cite{bauer1}, \cite{bauer2}, \cite{waterhouse}, \cite{hasse})
worked with ideals).
  Unlike
genus theory, however, neither the R\'edei algorithm nor the recursive one
imply
the existence of quadratic fields with prescribed $2$-class groups.
For example,
 Hajir points out that we do not know if there are
quadratic fields with arbitrarily large \textit{cyclic} $2$-class groups.
Such number
 fields are interesting because in general, if the $p$-class group of
a number field $K$ is cyclic then $p$ does not divide the class number of
the Hilbert class field of $K$; in particular, the $p$-class field tower of
$K$ is finite \cite[lemma 7]{hajir}.
In this paper we give an affirmative answer to Hajir's question.

\vsp

\begin{thmm}
	\label{thm:main}
For any integer $k\ge 1$, there exist infinitely many complex quadratic
fields
for which the Sylow $2$-subgroups of their class groups are cyclic of order
$
2^k
$.
\end{thmm}

\vsp

By modifying a standard argument, we can construct complex quadratic fields with arbitrarily large cyclic $2$-class group
by  finding distinct odd primes $p_1, p_2$ whose sum is four times
 an even integer
power, cf.~corollary \ref{cor:two}.  The existence of such prime pairs
is guaranteed by a theorem of Perelli, which says that the binary Goldbach
problem for values of polynomials is true on average \cite{perelli}.
To pin down the exact size of the cyclic $2$-class groups so produced, we
 apply the aforementioned  criterion for deciding if an ideal class is a
square to these prime pairs and turn it into congruence conditions on the
prime pairs.  Mimicing Perelli's circle method argument with these additional
congruence conditions and the theorem follows. We give the class group arguments in \S\ref{sec:cyclic}. The proof of Theorem \ref{thm:main} is completed by generalizing some circle method results of Perelli, which we do in \S\ref{sec:circlemethod}.

\vsp

\begin{remark}
We do not have an analogous result  for \textit{real}
quadratic fields, and there remains the question of constructing infinitely many quadratic class
groups, real or complex, with prescribed, non-cyclic (narrow or full)
$2$-class groups.
\end{remark}

%


\section{Prescribed cyclic $2$-class groups}
    \label{sec:cyclic}

\begin{lem}
	\label{lem:class}
Fix an integer $m>1$, and set $d = 4 w^{2m} - x^2>0$ with
$
w, x\in\Z
$
and positive.
If $w$ is even, $(x,w)=1$, and
$
0 < x \le 2 w^m-2
$,
then the class group of
$
\Q(\sqrt{-d})
$
contains an element of order $2m$.
\end{lem}

\begin{proof}
This is classical,  see for instance \cite{ankeny}, except we do not require
$w$ to be prime and we need $d$ to be odd.  For completeness we give the
argument.

Since $x$ is odd and $-d\equiv\mymod{1}{4}$, both
$
\frac{x \pm \sqrt{-d}}{2}
$
are algebraic integers in the ring of integers
$\oo_{-d}$ of $\Q(\sqrt{-d})$.  We claim that
the principal ideals $(\frac{x\pm\sqrt{-d}}{2})$  are coprime;
 otherwise some prime ideal
$\p$ of $\oo_{-d}$ divides $(\frac{x^2+d}{4}) = (w^{2m})$ as well as the
\textit{element}
$
\frac{x+\sqrt{-d}}{2} +  \frac{x- \sqrt{-d}}{2} = x
$,
whence $\p$ divides $(w, x)=1$, a contradiction.
 Thus
$
(\frac{x\pm\sqrt{-d}}{2})
$
are
coprime ideals, whence the equality
$
4w^{2m}  = d + x^2
$
implies that
$
(\frac{x+\sqrt{-d}}{2}) = J^{2m}
$
for some ideal $J$ of $\oo_{-d}$ with $J\not=J' :=$ the conjugate of $J$.
Moreover,
$
\text{Norm}(J) = w
$
since $w>0$.

Suppose the ideal class of $J$ has order $< 2m$, and hence a proper divisor
of $2m$.  Then
$J^n$
is a principal ideal $(u+v\sqrt{-d})$ for some $0<n\le m$ with
$
u, v\in \frac{1}{2}\Z
$.
If $v\not=0$, then
$$
d/4\ \le\ u^2+dv^2\ =\ \text{Norm}_{k/\Q}(J^n).
$$
Since $n\le m$, we have
$
\text{Norm}_{k/\Q}(J^n) = w^n \le w^m
$,
so
$
d \le 4w^{m}
$.
But
$
0<d=4w^{2m}-x^2
$,
so
$
4w^{2m} - 4w^m \le x^2
$,
whence
$
(2w^m - 1)^2 \le x^2 + 1 < (x+1)^2
$,
contradicting our hypothesis on $x$.  Thus $v=0$, whence
$
J^n = (u) = (J')^n
$
as ideals,
and hence $J=J'$, a contradiction.  So the ideal class of $J$ has order
 exactly $2m$, and the lemma follows.
\end{proof}

\vsp

\begin{cor}
    \label{cor:two}
Fix an integer $k\ge 1$, and suppose that there exists an
 even integer $w$
such that
$
4 w^{2^{k-1}}
$
is the sum of two distinct primes $p_1, p_2 \ge 3$.  Then the $2$-class group of
$
\Q(\sqrt{-p_1 p_2})
$
is cyclic of order divisible by $2^k$.
\end{cor}

\begin{proof}
Since $w$ is even, necessarily $p_1 p_2 \equiv\mymod{3}{4}$.  Then
genus theory implies that the $2$-class group of
$
\Q(\sqrt{-p_1 p_2})
$
is cyclic.  We can write the two primes as
$
2 w^{2^{k-1}} \pm x
$
with $0 < x \le 2w^{2^{k-1}}-2$, so
$
p_1 p_2 =  4 w^{2^{k}} - x^2 = 4 w^{2\cdot 2^{k-1}} - x^2
$;
the condition $0<x\le 2w^m - 2$ now becomes $p_i\ge 3$.
Apply Lemma \ref{lem:class} and we get the second part of the corollary.
\end{proof}

\vsp

For the rest of this paper, we take
$$
p_1 p_2 = 4w^{2^{k}}-x^2
\:\:
\text{ with $k\ge 1, w$ even and positive, and $p_1, p_2 \ge 3$ distinct primes.}
$$
The condition that  $w$ is even implies that $p_1 p_2 \equiv\mymod{3}{4}$; from
 now on we stipulate that
$$
p_1 \equiv\mymod{1}{4}
\:\:
\text{ and }
p_2 \equiv\mymod{3}{4}.
$$
Lemma \ref{lem:class} then furnishes an ideal $J$ in
$\oo_{-d}$
of norm $w$
whose ideal class has exact order $2^k$.  We now determine whether or not
this ideal class of $J$ is not the square of another ideal class.    Since
the $2$-class group of $\oo_{-d}$ is cyclic, this is equivalent to asking if
the $2$-class group has exact order $2^k$.

\vsp

\begin{lem}
       \label{lem:symbol}
With the notation as above,  the $2$-class group of
$
\Q(\sqrt{-p_1 p_2})
$
is cyclic of exact order $2^k$ if and only if the
quadratic symbol
$
( p_1/ w ) = -1
$.
\end{lem}

\begin{proof}
For any negative fundamental discriminant $-D$ and for any ideal
$
\mathfrak a\subset\oo_{-D}
$,
the `fundamental criterion' in \cite[p.~345]{hasse} says that the ideal class
of $\mathfrak a$ is a square if and only the Hilbert symbols
$$
\Bigl(  \frac{ \text{Norm}(\mathfrak a), -D}{p} \Bigr) = 1
	\hspace{20pt}
\text{for every  prime $p|D$.}
$$
Moreover, by \cite[(1.8) on p.~345]{hasse},
\begin{equation}
\prod_{p|D} \Bigl(  \frac{ \text{Norm}(\mathfrak a), -D}{p} \Bigr)
	= 1.
\label{all}
\end{equation}
We now apply this to the ideal class $J$ of order divisible by $2^k$
furnished by Lemma
\ref{lem:class}.  We have
 Norm$(J)=w$,
and since the fundamental discriminant
$
-p_1 p_2
$
has exactly two prime divisors, (\ref{all}) implies that
$$
\Bigl(  \frac{ w, -p_1 p_2}{p_1} \Bigr)
=
\Bigl(  \frac{ w, -p_1 p_2}{p_2} \Bigr).
$$
In particular,
$$
\text{the ideal class of $J$ is \textit{not} a square}
	\Longleftrightarrow
\Bigl(  \frac{ w, -p_1 p_2}{p_1} \Bigr) = -1
	\Longleftrightarrow
\Bigl(  \frac{ w, -p_1 p_2}{p_2} \Bigr) = -1.
$$
Since
$
4w^{2^{k-1}} = p_1 + p_2
$,
we have
$
p_1\nmid w
$.
Expressing the Hilbert symbol in terms of Legendre symbols
\cite[Thm.~1 on p.~20]{serre}
and recalling that
$
p_1\equiv\mymod{1}{4}
$
and
$
p_2\equiv\mymod{3}{4}
$,
we see that
\begin{equation}
\text{the ideal class of $J$ is \textit{not} a square}
	\Longleftrightarrow
\Bigl( \frac{p_1}{w} \Bigr) = -1
	\Longleftrightarrow
\Bigl( \frac{p_2}{w} \Bigr) = \Bigl(\frac{-1}{w}\Bigr).
	\label{final}
\end{equation}
As the $2$-class group of $\Q(\sqrt{-p_1 p_2})$ is cyclic,
we are done.
\end{proof}

\section{Circle Method}\label{sec:circlemethod}

Thanks to Lemma \ref{lem:symbol}, to prove Theorem \ref{thm:main} for a
given $k\ge 1$
 we need to find infinitely many pairs of distinct odd primes
$p_1, p_2\ge 3$ such that
\begin{itemize}
\item[(i)]
  $p_1 + p_2 =4 w^{2^{k-1}}$ with $w$ even, and
\item[(ii)]
$p_1\equiv\mymod{1}{4}$, and
\item[(iii)]
$(p_1/w)=-1$.
\end{itemize}
\noindent
The first condition is the binary Goldbach problem for polynomials,
which Perelli \cite{perelli} has already shown to be true on average.
Denote by $\Lambda(d)$ the von Mangoldt function, and define
$$
R(d) = \sum_{d_1 + d_2 = d, \: d_i>0} \Lambda(d_1) \Lambda(d_2).
$$
Let $F\in\Z[x]$ be a non-constant polynomial.   For any integer $d>0$,
define
$$
\rho_F(d) = \# \{ \mymod{m}{d}: F(m)\equiv\mymod{0}{d} \}.
$$
Denote by $a(F)$ the leading coefficient  of $F$.  Define
$$
C(F) = a(F) \rho_F(2) \prod_{p>2}
  \Bigl(
    1 + \frac{\rho_F(p)}{p(p-2)}
  \Bigr)
  \Bigl(
    1 - \frac{1}{(p-1)^2}
  \Bigr).
$$
Suppose $a(F)$ is \textit{positive}.  For $N>0$, set
$
N_F = N^{1/{(2\deg F)}}
$.
A special case of a theorem of Perelli \cite[Thm.~2]{perelli} says that for
any $A>0$,
\begin{equation}
\sum_{N_F^2\le n \le N_F^2 + N_F}  R(F(n))
=
C(F) N_F^{1+2\deg F} + O_{A, F}(N^{1+2\deg F}\log^{-A}\! N_F).
       \label{perelli}
\end{equation}
If $d$ is not the sum of two primes then
$$
R(d)
\le
\sum_{p^n\le d, \: n\ge 2}  \log(p)\log (\sqrt{d-p^n})
\ll
d^{1/2}\log^2 d,
$$
so the contribution to the left side of (\ref{perelli}) from those $F(n)$
that are not the sum of two primes is
$
O_F(N_F^{1 + \deg F}\log^2\! N_F)
$.
Similarly, the contribution to the left side of (\ref{perelli})
 from those $n$
for which $F(n) = p_1 + p_2$ with one of the $p_i<5$ is
$
O( N_F)
$.
Since $F$ is non-constant, $C(F)\not=0$ if and only if
$
\rho_F(2)\not=0
$,
which in turn is equivalent to $F$ taking on even values.  So Perelli's
theorem implies that if
$
F\in\Z[x]
$
is non-constant and takes on even values, then infinitely many
of its values can be written as the sum of two primes $\ge 3$.
Applying this to
$
F(x) = 2(2x)^{2^{k-1}},
$
we see that for any fixed $k>1$, Perelli's theorem implies that
 condition (i) holds for  infinitely many pairs of primes $p_1, p_2 \ge 3$.

\vsp

Perelli's theorem is proved using the circle method, for which we can readily
introduce congruence conditions such as (ii).  We now explain how to handle
condition (iii).  Suppose the even integer $w$ is of the form
$
w = 2M ^2
$
for some integer $M$.  Note that $w$ is coprime to the $p_i$, so
$$
\leg{p_1}{w}
\:=\:
\displaystyle
\leg{p_1}{2} \leg{p_1}{M}^{\!2}
\\
\:=\:
\displaystyle
\leg{p_1}{2}
\\
\:=\:
\displaystyle
\Bigl\{
  \begin{array}{rl}
     1 & \mbox{if $p_1 \equiv \pm 1 \pmod 8$}
     \\
     -1 & \mbox{if $p_1 \equiv \pm 3 \pmod 8$.}
  \end{array}
$$
Combine everything and we see that to prove Theorem \ref{thm:main} it suffices
to prove the following result.

\vsp

\begin{thm}
  \label{thm:circlemethod}
Given any $k\ge 1$,
there exist infinitely many pairs of primes $p_1, p_2 \ge 3$ and
integers $w=2M^2$ such that
\begin{enumerate}
\item
$
p_1 + p_2 = 4 (2M^2)^{2^{k-1}} = 2^{1+2^{k-1}} M^{2^k}
$
for some integer $M$, and
\item
  $p_1 \equiv 5 \pmod 8$ and $p_2\equiv\mymod{3}{8}$.
\end{enumerate}
\end{thm}


There are several different approaches we could take to prove Theorem \ref{thm:circlemethod}. We chose to modify Perelli's paper by changing the generating function and mirroring the calculations. Another approach would be to introduce products of characters to restrict to given equivalence classes (in the spirit of \eqref{formula} below) earlier. Such an attack would quickly give us the factor of 1/4 for the main term, but would require a new analysis of bounds on exponential sums twisted by a character. Further, it would still require us to mirror Perelli's arguments. As both approaches require us to follow Perelli's paper, to aid the reader who may not be as familiar with these techniques as with the earlier algebraic arguments, we take the first approach.

\subsection{Preliminaries}

Let $N$ be a large integer, and $F \in \mathbb{Z}[x]$ with $\deg F = k \geq 1$.
Define
\begin{equation}
 f_2(\alpha;N) = \sum_{p \leq c_1 N} (\log p) e(\alpha p)
\end{equation}
where $p$ ranges over primes congruent to $\pm 3 \pmod 8$, $c_1$ is a
suitable coefficient depending on $F$, and as always \be e(x) \ = \ e^{2\pi i x}.\ee  We choose this notation to mirror that of Vaughan
\cite[p.~27-37]{vaughan}, where $f(\alpha)$ is the same sum, but without the restriction on $p$.  We
use the subscript of 2 in function definitions throughout to point out such similar parallels.

We briefly comment on the role $c_1$ plays. In our applications we will take $n$ satisfying $N^{1/k} \le n \le N^{1/k} + H$ for some small $H \le N^{1/k}$. Thus if $F(x) = cx^k + \cdots$ then $F(n)$ will be on the order of $cN$. If $c=3$ then we cannot write $F(n) = p_1 + p_2$ if we restrict to primes at most $N$. This is why we must let the sums be a little longer; taking $c_1$ to be $3^k$ times the leading coefficient of $F$ will clearly suffice for all large $N$.

Define $P = (\log N)^B$ where $B$ is a positive constant, and for $1
\leq a \leq q \leq P$ with $(a,q)=1$ define
\begin{equation}
 \mathfrak{M}(q,a) = \{ \alpha : |\alpha - a/q| \leq P N^{-1}\}
\end{equation}
as the major arc centered at $a/q$, and let $\mathfrak{M}$ denote the union
of all the $\mathfrak{M} (q,a)$. Since $N$ is large, the major arcs are all
disjoint and lie in $(P N^{-1}, 1 + P N^{-1}]$.  We define the minor arcs by
\[
\mathfrak{m} = (P N^{-1}, 1 + P N^{-1}]\setminus\mathfrak{M}.
\]
We have
\begin{equation}
 R_2 (n) = \int_{\mathfrak{M}} f_2(\alpha;N)^2 e(-n\alpha) \, d\alpha +
\int_{\mathfrak{m}} f_2(\alpha;N)^2 e(-n\alpha) \, d\alpha
\end{equation}
where
\begin{equation}
 R_2 (n) = \sum_{\substack{p_1, p_2 \le c_1 N \\ p_1 + p_2 = n}} \log p_1 \ \log p_2
\end{equation}
with, again, the primes restricted mod 8 as with $f_2(\alpha;N)$. Note $R_2(n)$ is a weighted counting of the number of representations of $n$ as a sum of two primes congruent to 3 or 5 modulo 8. Clearly if $R_2(n) > 0$ then there must be at least one representation of $n$ having the desired properties. Further, from our choice of $c_1$ we see every representation of $n$ is counted if $n$ is of size $N + o(N)$; if $n$ is much larger, we are only counting the representations with each prime factor at most $c_1N$ and could therefore miss some.

While we are able to obtain good formulas for the integral over the major arc for a fixed $n$, there are no correspondingly nice estimates for the minor arc contribution. This should not be surprising, as if there were then we would essentially be solving the original Goldbach problem! Rather, we lower our expectations and obtain a good upper bound for the sum of the contributions of the minor arcs for all $n$ in a small interval. The advantage of this approach is that we can now exploit cancellation from the $n$-sum; without this cancellation we have no chance of proving our claim.

We first state some needed arithmetic input in \S\ref{sec:arithminput}. We then analyze the major arcs in \S\ref{sec:contrmajorarcs}. In \S\ref{sec:proofthmcirclemethod} we prove Theorem \ref{thm:circlemethod}, in the course of which we bound the contribution from the minor arcs.

\subsection{Arithmetic Input}\label{sec:arithminput}

The following function is used throughout our analysis of the major and minor arcs.

\begin{defn}
 Let $\mu_2$ be an arithmetic function defined by $\mu_2 (q) = \mu(q)/2$
whenever $8 \nmid q$, $\mu_2 (8) = -\sqrt{2}$ and
\begin{equation}
  \mu_2 (8q)\ =\ \sum_{\substack{r = 1 \\ (r,8q) = 1}}^{8q} e(r/8q)
\end{equation}
where $r\equiv \pm 3 \pod 8$.
\end{defn}

Note that without the latter restriction the sum is just
$\mu (8q)=0$ by a well-known theorem.  Hence if $\mu_3$ is the same sum but
instead with the restriction that $r \equiv \pm 1 \pod 8$, then $\mu_2(8q) =
-\mu_3(8q)$. The next lemma is useful in finding a simple expression for $\mu_2(8q)$, which we will do in Lemma \ref{lem:formulamu28q}.

\begin{lem}\label{lem:neededlemformu2} Let $k \geq 3$ be an integer and $q$ be an odd positive integer.
Define $\phi_2 (2^k q)$ to be the number of positive integers $r < 2^k q$ with
$(r, 2^k q) = 1$ and $r \equiv \pm 3 \pod 8$.  Define $\phi_3 (2^k q)$ similarly
except with $r \equiv \pm 1 \pod 8$.  Then
\[
\phi_2 (2^k q) \ = \ \phi_3 (2^k q) \ = \ 2^{k-2} \phi(q).
\]
\end{lem}
\begin{proof}
Clearly if $(r, 2^k q) = 1$, then either $r \equiv \pm 3 \pod 8$ or $r\equiv \pm
1 \pod 8$, so we have $\phi_2 (2^k q) + \phi_3 (2^k q) = \phi(2^k q) =
2^{k-1} \phi(q)$.  So we need only show $\phi_2 (2^k q) = \phi_3 (2^k q)$.

Let $a$ and $b$ be any integers congruent to 3 mod 8 with $(a, 2^k q) = 1 = (b,
2^k q)$.  Then letting $c$ and $d$ be the smallest positive integers congruent
to $a + 4 q$ and $b+4q$, respectively, we quickly have $(c, 2^k q) = 1 = (d, 2^k
q)$, $c \equiv d \equiv 7 \pod 8$, and $c = d \implies a = b$.  Using a similar
argument to map integers 5 mod 8 to those 1 mod 8, we have that $\phi_2
(2^k q)\leq \phi_3 (2^k q)$.  By a near-identical argument, we have that $\phi_3
(2^k q) \leq \phi_2 (2^k q)$.  Hence $\phi_2 (2^k q) = \phi_3 (2^k q)$, and the
lemma follows.
\end{proof}

We now give a simple formula for $\mu_2(8q)$, which will be useful in evaluating certain exponential sums.

\begin{lem}\label{lem:formulamu28q} We have
\begin{equation}
 \mu_2 (8q) \ = \ \leg{q}{2} |\mu (q)| \sqrt{2}
\end{equation}
for all $q > 1$, where $\leg{q}{2}$ is the Kronecker symbol
\begin{equation*}
 \leg{q}{2} \ = \ \begin{cases}
               0 & \mbox{{\rm if} $q$ {\rm is even},}\\
               1 & \mbox{{\rm if} $q\equiv \pm 1 \pmod 8$,}\\
               -1 & \mbox{{\rm if} $q \equiv \pm 3 \pmod 8$.}
              \end{cases}
\end{equation*}
\end{lem}
\begin{proof}
 First note that
\begin{align*}
\mu_2 (8) &\ = \ e(3/8)+e(5/8)\\
&\ = \ 2\cos (3\pi/4)\\
&\ = \ -\sqrt{2}.
\end{align*}
It happens that for all $k > 3$, $\mu_2 (2^k) = 0$.  To see this, first note
that for $r \equiv \pm 3 \pod 8$, $(r, 2^k) = 1$.  So
\[
\mu_2 (2^k) \ = \ \sum_{n=1}^{2^{k-3}} e\left( \frac{8n-5}{2^k} \right) +
\sum_{n=1}^{2^{k-3}} e\left( \frac{8n-3}{2^k} \right).\\
\]
As each of these sums is geometric with common ratio $e(8/2^k) =
e(1/2^{k-3})$, we have
\begin{align*}
\mu_2 (2^k) &\ = \ \frac{e\left( \frac{2^k+3}{2^k} \right)-e\left( \frac{3}{2^k}
\right)}{e(1/2^{k-3})-1} + \frac{e\left( \frac{2^k+5}{2^k} \right)-e\left(
\frac{5}{2^k} \right)}{e(1/2^{k-3})-1}\\
&\ = \ 0.
\end{align*}
We now show that, for odd $q$ and any $k \geq 3$, $|\mu_2 (2^k q)| = |\mu_2
(2^k) \mu(q)|$. We have from Lemma \ref{lem:neededlemformu2} that $\mu_2 (2^k q)$ (resp. $\mu_3 (2^k q)$) consists of
a sum of $2^{k-2} \phi(q)$ terms ranging over all $r \equiv \pm 3 \pod 8$ (resp.
$r \equiv \pm 1 \pod 8$) such that $(r, 2^k q) = 1$.  Using the fact that
\[
\mu(q) \ = \ \sum_{\substack{r = 1 \\ (r,q) = 1}}^q e(r/q),
\]
a sum with $\phi(q)$ terms, we note that a sum expansion of $\mu_2 (2^k) \mu(q)$
would contain $2^{k-2} \phi(q)$ terms.  Each term would be of the form $e(s/2^k)
e (t/q) = e\left( \frac{sq+2^k t}{2^k q} \right)$ with $s \equiv \pm 3 \pod 8$
and $(t,q) = 1$.  Clearly this fraction is always in lowest terms, since $2
\nmid sq$ and $(2^k t, q) = 1$.  Depending on whether $q \equiv \pm 1 \pod 8$ or
$\pm 3 \pod 8$, $sq+2^k t$ is always either $\pm 3 \pod 8$ or $\pm 1 \pod 8$,
respectively.  Also, if $s_0 q + 2^k t_0 \equiv s_1 q + 2^k t_1 \pmod {2^k q}$,
then $(s_0 - s_1) q + (t_0-t_1) 2^k \equiv 0 \pmod {2^k q}$ and so $s_0 = s_1$
and $t_0 = t_1$ since $s_0, s_1 < 2^k$ and $t_0, t_1 < q$.  Therefore the
expansion of $\mu_2 (2^k) \mu(q)$ contains exactly $2^{k-2} \phi(q)$ different
terms, and either $\mu_2 (2^k) \mu(q) = \mu_2 (2^k q)$ or $\mu_2 (2^k) \mu(q) =
\mu_3 (2^k q)$.  Since $|\mu_2| = |\mu_3|$, we can say that $|\mu_2 (2^k)
\mu(q)| = |\mu_2 (2^k q)|$.

From here it's not hard to check that $\mu_2 (2^k q) = -\sqrt{2}$ if $q>1$ is
squarefree and $\pm 3$ mod 8 and $\mu_2 (2^k q) =\sqrt{2}$ if $q>1$ is
squarefree and $\pm 1$ mod 8.
\end{proof}

We end with the promised relation between an exponential sum and $\mu_2$, which is used in finding a good approximating function to $f_2$ on the major arcs.

\begin{lem}\label{lem:neededrelexpsummu2} With the usual restrictions on $r$ and $(a,8q) = 1$, we have
\begin{equation}
  \sum_{\substack{r = 1 \\ (r,8q) = 1}}^{8q} e(ar/8q) \ = \ \leg{a}{2} \mu_2 (8q).
\end{equation}
\end{lem}

\begin{proof}
This follows immediately from the arguments above, as $\mu_2 (r)$ if
$a \equiv \pm 1 \bmod 8$ and and $\mu_3(r) = -\mu_2 (r)$ if $a \equiv \pm 3 \bmod
8$.
\end{proof}

\subsection{Contribution from the Major Arcs}\label{sec:contrmajorarcs}

The following analogue of Vaughan's Lemma 3.1 \cite[p.~30]{vaughan}
allows us to approximate our generating function $f_2$ on the major arcs with a well-behaved function, up to a small error.

\begin{lem}\label{lem:vapproxf2}
 Let
\begin{equation}
 v(\beta) \ = \ \sum_{m=1}^{c_1N} e(\beta m).
\end{equation}
Then there is a positive constant $C$ such that whenever $1 \leq a \leq q \leq
P$, $(a,q) = 1$ and $\alpha \in \mathfrak{M}(q,a)$ one has
\begin{equation}
 f_2(\alpha;N) \ = \ \frac{\mu_2 (q)}{\phi (q)} v(\alpha-a/q) + O(N \exp(- C (\log
n)^{1/2}))
\end{equation}
whenever $8 \nmid q$ and
\begin{equation}
 f_2(\alpha;N) \ = \ \frac{\leg{a}{2} \mu_2 (q)}{\phi (q)} v(\alpha-a/q) + O(N \exp(-
C (\log
n)^{1/2}))
\end{equation}
otherwise.
\end{lem}

\begin{proof}
Following Vaughan \cite{vaughan}, we define our generating function to be
\begin{equation}\label{eq:defnf2alpha}
 f_2(\alpha;N) \ = \ \sum_{\substack{p \leq c_1N \\ p \equiv \pm 3 \, (8)}} (\log p)
e(\alpha p).
\end{equation} We note that the corresponding generating function for Perelli \cite{perelli} is (in his notation) \be S(\alpha) \ = \ \sum_{n \le c_1 N} \Lambda(n) e(n\alpha).\ee There is no harm in replacing $\Lambda(n)$ (which restricts the sum to prime powers) with $\lambda(n)$ (which restricts the sum to primes), as the difference between these two sums is $O(N^{1/2})$, which can easily be absorbed by the error terms. In fact, for our purposes it is better to restrict to prime sums as this way we know that the representation will be as a sum of two primes and not potentially a prime power.

We first consider the special case $\alpha = a/q$, and then as is standard pass to all $\alpha$ in the major arc $\mathfrak{M}(q,a)$. We have
\begin{equation}
 f_2(a/q ;N) \ = \ \sum_{\substack{r = 1 \\ (r,q) = 1}}^q e(ar/q) \vartheta_2
(c_1N,q,r) + O((\log N)(\log q))
\end{equation}
with
\begin{equation}
 \vartheta_2 (x,q,r) \ = \ \sum_{\substack{p \leq x\\ p \equiv r\, (q)}} \log p
\end{equation}
where the primes $p$, in addition to satisfying $p \equiv r \bmod q$, must also
satisfy $p \equiv \pm 3 \bmod 8$.  We must now break into cases depending on $q$.

If $q = 2^k q_0$ with $k \leq 1$ and $q_0$ odd, then $p \equiv r \bmod q$ and $p
\equiv \pm 3 \bmod 8$ imply that there exist $r_1$ and $r_2$ such that
\begin{align}
 \vartheta_2 (c_1N,q,r) &\ = \ \sum_{\substack{p \leq c_1N \nonumber\\ p \equiv r_1\, (8q_0)}} \log
p + \sum_{\substack{p \leq c_1N \\ p \equiv r_2\, (8q_0)}} \log p \nonumber\\
 &\ = \ \frac{2 c_1N}{\phi(8 q_0)} + O(N \exp(- C_1 (\log N)^{1/2}))\nonumber\\
 &\ = \ \frac{c_1N}{2\phi(q)} + O(N \exp(- C_1 (\log N)^{1/2}))
\end{align}
by the Siegel-Walfisz theorem\footnote{The Siegel-Walfisz theorem states that for $C, B > 0$ and $a$ and $q$ relatively prime, then  $\sum_{ {p \le x} \atop {p \equiv a (q)} } \log p \ = \
\frac{x}{\phi(q)} + O( x/\log^C x)$ for $q
\le \log^B x$, and the big-Oh constant depends only on $C$ and $B$. See for example \cite{davenport,IK}.
} and the fact that, if $q = 2q_0$, $\phi(q) =
\phi(q_0)$.

If $q = 4q_0$ with $q_0$ odd, then $p \equiv r \bmod q$ and $p \equiv \pm 3
\bmod 8$ implies that either all the $p$ are $3 \bmod 8$ or $5 \bmod 8$
(depending on $r \bmod 4$), so there exists $r_1$ such that
\begin{align}
 \vartheta_2 (c_1N,q,r) &\ = \ \sum_{\substack{p \leq c_1N \\ p \equiv r_1\, (8q_0)}}
\log p \nonumber\\
 &\ = \ \frac{c_1N}{\phi(8 q_0)} + O(N \exp(- C_2 (\log N)^{1/2}))\nonumber\\
 &\ = \ \frac{c_1N}{2\phi(q)} + O(N \exp(- C_2 (\log N)^{1/2})).
\end{align}

Hence, if $8 \nmid q$,
\begin{equation}
 f_2(a/q ;N) \ = \ \frac{c_1N}{2\phi(q)} \sum_{\substack{r = 1 \\ (r,q) = 1}}^q
e(ar/q) + O(N \exp(- C (\log N)^{1/2})).
\end{equation}
The sum on the right hand side is well-known to equal $\mu (q)$.  Therefore
\begin{align}
 f_2(a/q ;N) &\ = \ \frac{c_1N \mu(q)}{2\phi (q)} + O(N \exp(- C (\log N)^{1/2}))\nonumber\\
  &\ = \ \frac{c_1N\mu_2 (q)}{\phi(q)} + O(N \exp(- C (\log N)^{1/2})).
\end{align}

When $8 \mid q$, we have that $\vartheta_2 (c_1N,q,r) = 0$ if $r \equiv \pm 1
\bmod 8$ and
\begin{equation}
 \vartheta_2 (c_1N,q,r) \ = \ \frac{c_1N}{\phi(q)} + O(N \exp(- C_3 (\log N)^{1/2}))
\end{equation}
if $r \equiv \pm 3 \bmod 8$.  Hence
\begin{equation}
 f_2(a/q ;N) \ = \ \frac{c_1N}{\phi(q)} \sum_{\substack{r = 1 \\ (r,q) = 1}}^q
e(ar/q) + O(N \exp(- C (\log N)^{1/2}))
\end{equation}
where the sum only counts $r \equiv \pm 3 \bmod 8$.  However, we know from
Lemma \ref{lem:neededrelexpsummu2} that this sum is just $\leg{a}{2}\mu_2 (q)$.  So when $8 \mid q$,
\begin{equation}
 f_2(a/q ;N) \ = \ \frac{c_1N\leg{a}{2}\mu_2 (q)}{\phi(q)} + O(N \exp(- C (\log
n)^{1/2})).
\end{equation}
As $v(0) = c_1N$, the lemma's claim is true when we take $\alpha = p/q$.
The rest of the proof (namely, general $\alpha$ in the major arc $\mathfrak{M}(q,a)$) proceeds almost identically to the proof of the corresponding lemma in Vaughan \cite[p.~31]{vaughan},
 with the only difference being 2-subscripts added in obvious places.
\end{proof}

Consider a major arc $\mathfrak{M}(q,a)$ where $8\mid q$.  Since $(q,a) = 1$, we
must have that $a$ is odd, and hence $\leg{a}{2}^2 = 1$.  Therefore, for
$\alpha \in \mathfrak{M}(q,a)$, an arbitrary major arc, the lemma above gives
\begin{equation}
 f(\alpha)^2 - \frac{\mu_2(q)^2}{\phi (q)^2} v(\alpha-a/q)^2 \ll N^2 \exp(-C
(\log N)^{1/2}).
\end{equation}

The rest of the major arc treatment is the same as in Vaughan (p. 31-32),
except for one non-trivial adjustment.  The original singular series
\begin{equation}
 \mathfrak{S}_1 (m) \ = \ \sum_{q=1}^\infty \frac{\mu(q)^2}{\phi(q)^2}
\sum_{\substack{a=1\\ (a,q)=1}}^q e(-am/q)
\end{equation}
needs to be replaced with
\begin{equation}
 \mathfrak{S}_2 (m) \ = \ \sum_{q=1}^\infty \frac{\mu_2(q)^2}{\phi(q)^2}
\sum_{\substack{a=1\\ (a,q)=1}}^q e(-am/q),
\end{equation}
and we must show that $\mathfrak{S}_2(m)$ can also be bounded away from zero.
We have that
\begin{align}
 \mathfrak{S}_2 (m) &\ = \ \frac{1}{4}\sum_{\substack{q=1\\ 8\nmid q}}^\infty
\frac{\mu(q)^2}{\phi(q)^2} c_q (m) + \sum_{q=1}^\infty
\frac{\mu_2(8q)^2}{\phi(8q)^2} c_{8 q} (m)\nonumber\\
&\ = \ \frac{\mathfrak{S}_1 (m)}{4} + \frac{1}{8}c_8 (m)\sum_{\substack{q=1\\
q\text{ odd}}}^\infty \frac{\mu(q)^2}{\phi(q)^2} c_q (m) \label{qodd}
\end{align}
where the $8 \nmid q$ restriction in the first sum can be ignored since
$\mu (8q)=0$ anyway, and $c_q (m)$ is Ramanujan's sum
\begin{equation}
 c_q (m) \ = \ \sum_{\substack{a=1\\ (a,q)=1}}^q e(-am/q),
\end{equation}
which is a multiplicative function of $q$, and is known to equal
\begin{equation}
 c_q (m) \ = \ \frac{\mu(q/(q,m)) \phi(q)}{\phi(q/(q,m))}.
\end{equation}
  To
calculate the sum in \eqref{qodd} over $q$ odd, first write
\begin{equation}
\mathfrak{S}_1 (m) \ = \ \sum_{\substack{q=1\\ q\text{
odd}}}^\infty \frac{\mu(q)^2}{\phi(q)^2} c_q (m) + \sum_{\substack{q=2\\ q\text{
even}}}^\infty \frac{\mu(q)^2}{\phi(q)^2} c_q (m).
\end{equation}
Since $\mu(q) = 0$ whenever $4 \mid q$, we can take the second sum just
over $q = 2q_0$, where $q_0$ is odd.  Since $\mu,\phi,$ and $c_q$ are all
multiplicative on $q$, we have
\begin{align}
 \mathfrak{S}_1 (m) &\ = \ \sum_{\substack{q=1\\ q\text{
odd}}}^\infty \frac{\mu(q)^2}{\phi(q)^2} c_q (m) +
\frac{\mu(2)^2}{\phi(2)^2} c_2(m) \sum_{\substack{q=1\\ q\text{
odd}}}^\infty \frac{\mu(q)^2}{\phi(q)^2} c_q (m)\nonumber\\
&\ = \ 2\sum_{\substack{q=1\\ q\text{
odd}}}^\infty \frac{\mu(q)^2}{\phi(q)^2} c_q (m).
\end{align}
Therefore
\begin{equation}
 \mathfrak{S}_2 (m)\ = \ \frac{\mathfrak{S}_1 (m)}{4} \left(1 + \frac{c_8
(m)}{4}\right).
\end{equation}
Note that $\mathfrak{S}_2(m) = 0$ when $m$ is odd or $m \equiv 4 \mod 8$, which
is what one would expect, since the only numbers which can be written as the sum
of primes 3 and/or 5 mod 8 are those 0, 2, and 6 mod 8, and in those cases it
is clear that $\mathfrak{S}_2 (m) \gg 1$ since $\mathfrak{S}_1 (m) \gg 1$.


\subsection{Proof of Theorem \ref{thm:circlemethod}}\label{sec:proofthmcirclemethod}

We are now ready to prove the following result, which we then show immediately yields Theorem \ref{thm:circlemethod}.

\begin{thm}\label{thm:thmneededforproofmainresult} Let $F \in \Z[x]$ be a polynomial of degree $k > 0$ that is not always odd, $L = \log N$, $A,\varepsilon \geq 0$, and assume $H$ satisfies $N^{1/(3k)+\varepsilon} \leq H
\leq N^{1/k-\epsilon}$. Then
 \begin{equation}
  \sum_{N^{1/k} \leq n \leq N^{1/k}+H} |R_2 (F(n)) - F(n) \mathfrak{S}_2
(F(n))|^2 \ll H N^2 L^{-A}.
 \end{equation}\end{thm}

\begin{proof}
Taking $H = N^{1/(3k)+\varepsilon}$, $\varepsilon>0$ sufficiently small, we
may write
\begin{eqnarray}
& & \sum_{N^{1/k} \leq n  \leq N^{1/k}+H} |R_2(F(n)) - F(n) \mathfrak{S}_2
(F(n))|^2
\nonumber\\
& \ = \  & \sum_{N^{1/k} \leq n \leq N^{1/k}+H} \Bigg| \int_\mathfrak{M} f_2(\alpha;N)^2
e(-F(n)\alpha) \, d\alpha \nonumber\\ & & \ \ \ \ \ +\ \int_\mathfrak{m} f_2(\alpha;N)^2
e(-F(n)\alpha) \, d\alpha  - F(n) \mathfrak{S}_2 (F(n))\Bigg|^2\nonumber\\
& \le & \sum_{N^{1/k} \leq n \leq N^{1/k}+H} \left| \int_\mathfrak{M} f_2(\alpha;N)^2
e(-F(n)\alpha) \, d\alpha - F(n) \mathfrak{S}_2 (F(n))\right|^2\nonumber\\
& &  \ \ \ \ \ \ \ \ \ \ \ +\ \sum_{N^{1/k} \leq n \leq N^{1/k}+H} \left| \int_\mathfrak{m} f_2(\alpha;N)^2
e(-F(n)\alpha) \, d\alpha\right|^2 \nonumber\\ &\ = \ & \sum_\mathfrak{M} + \sum_\mathfrak{m}. \nonumber
\end{eqnarray}
We apply our estimation for $f_2(\alpha;N)$ from Lemma \ref{lem:vapproxf2} and
our new singular series $\mathfrak{S}_2$ (and its bounds) to integrate it over the major arcs, and then argue as in equations (2) to (4) of Perelli \cite{perelli} to bound its difference from $F(n) \mathfrak{S}_2 (F(n))$. We find
\begin{equation}
 \sum_\mathfrak{M} \ll H N^2 L^{-2B+c_2} + H N^2 L^{-A}
\end{equation}
where $c_2$ is a suitable constant based on $F$ and $L = \log N$.

We are left with bounding $\sum_\mathfrak{m}$. As this minor arc calculation is almost identical to that in \cite{perelli}, we simply highlight below the harmless modifications needed in adopting Perelli's framework to our problem. His minor arc calculation begins with equation (5) at the bottom of \cite[p.~481]{perelli}.
With the exception of the final reference in the last bullet point,
\textit{all} equation numbers and expressions in the bullet points
below
refer to items in \cite{perelli}. \\

\begin{itemize}

\item  Equation (5) follows from algebraic manipulation of $S$, and does not depend greatly on
the actual definition of $S$. We find a similar estimate for our quantity, defined in \eqref{eq:defnf2alpha}: \begin{equation}
 f_2(\alpha;N) \ = \ \sum_{\substack{p \leq c_1N \\ p \equiv \pm 3 \, (8)}} (\log p)
e(\alpha p),\nonumber
\end{equation} which differs from the sum in \cite{perelli} in two minor ways: we have $\lambda(n)$ instead of $\Lambda(n)$, and our $n$ is restricted to be $\pm 3 \bmod 8$. \ \\

\item Equations (6) through (8) prove a variant of Weyl's inequality, which is also true for our variant of $S$.\\

\item Equations (9) to (14) manipulate the arcs themselves and do not depend on $S$, and are therefore true in our case as well.\\

\item Equation (15) directly involves $S$, and therefore changes slightly in our case.  Remember that we want to restrict $S$ to only count primes restricted to those congruent to 3 and 5 mod 8.  Hence we replace $\mu$ with $\mu_2$ as defined in the major arc section, and we change $W$ to sum only over our restricted set of primes.  The change from $\mu$ to $\mu_2$ has absolutely no effect on the argument, since all that matters is that $\mu_2^2$ is bounded by a small constant (specifically, 2), which gets absorbed in the constant induced by the $\ll$. $T$ does not depend on primes, and is thus unchanged.\\

\item Equation (16) is a simple rearrangement, and follows for our problem as well.  Equations (17) to (23) are $L$-function arguments that do not change, and Equation (24) is just Parseval.\\

\item To get from Equation (24) to (25) requires an estimate for $W$, and therefore requires updating since we have changed the definition of $W$.  The estimation is based on the formula
\begin{equation}\label{formula}
 \sum_{n \leq x} \Lambda(n) \chi(n) - \delta_\chi x\ = \-\sum_\rho
\frac{x^\rho}{\rho} + O(L^2),\nonumber\
\end{equation}
where the sum on the right is over the non-trivial zeroes of $L(s,\chi)$ and
$\delta_\chi$ is 1 if $\chi = \chi_0$ and 0 otherwise.  We want to
restrict this sum to $n \equiv \pm 3 \bmod 8$. We can do this by writing
\begin{equation}
 \sum_{n \leq x} \Lambda(n) \chi(n) \frac{1-\leg{n}{2}}{2} - \delta_\chi x \ = \
\sum_{n \leq x} \Lambda(n) \chi(n) -\delta_\chi x - \frac{1}{2} \sum_{n\leq x}
\Lambda(n) \chi(n) \leg{n}{2}.\nonumber\
\end{equation}
Since $\chi(n) \leg{n}{2}$ is itself a non-principal character, we can
simply apply \eqref{formula} to the second sum, and we see that the restricted
sum is of the same order of magnitude as the original; hence we can conclude
the same estimation of the minor arcs.  The result follows directly from
Perelli's calculations. He cites work by himself and Pintz, \cite{perellipintz}. In that paper, we see the only change is in Equation (22), and that change is just the character manipulation which we've already described above.
\end{itemize}
\end{proof}

We can now complete the proof of our main result.

\begin{proof}[Proof of Theorem \ref{thm:circlemethod}]
By Theorem \ref{thm:thmneededforproofmainresult}, we have
 \begin{equation}
  \sum_{N^{1/k} \leq n \leq N^{1/k}+H} |R_2 (F(n)) - F(n) \mathfrak{S}_2
(F(n))|^2\ \ll\ \frac{H N^2}{\log^A N}\end{equation} with $N^{1/(3k)+\varepsilon} \leq H
\leq N^{1/k-\epsilon}$. Trivial estimation gives $F(n)\mathfrak{S}_2 (F(n)) \gg N$; here we are using $F(n) \gg N$ and $\mathfrak{S}_2(F(n)) \gg 1$. Imagine that none of the $F(n)$ (for $n$ in the range specified) can be represented as the sum of two primes satisfying our conditions. Then $R_2(F(n)) = 0$ for all these $n$, and the $n$-sum equals \be \sum_{N^{1/k} \leq n \leq N^{1/k}+H} |F(n) \mathfrak{S}_2
(F(n))|^2 \ \gg \ \sum_{N^{1/k} \leq n \leq N^{1/k}+H} N^2 \ = \ HN^2.\ee For $N$ sufficiently large, however, this contradicts Theorem \ref{thm:thmneededforproofmainresult}, which says the $n$-sum is at most $HN^2/\log^A N$. Thus there must be at least one $n$ in the given range that has the desired representation.

To obtain infinitely many $n$, we repeat these arguments on disjoint intervals; for example, if $N$ is sufficiently large we may take $N_\ell = N^\ell$, so for the $\ell$\textsuperscript{th} interval we have the range $N^{\ell/k} \le n \le N^{\ell/k} + H_\ell$.
\end{proof}

\begin{remark} In the interest of keeping the exposition as elementary as possible, we do not optimize the proof of Theorem \ref{sec:proofthmcirclemethod}. With a little more work, arguing along the lines of \cite{perelli} we could obtain some estimates on the number of $n$ in the studied intervals that have the desired representation. This would make our results on the infinitude of complex quadratic fields
for which the Sylow $2$-subgroups of their class groups are cyclic of order $2^k$ explicit, giving a lower bound on the number of such fields. As this would be at the cost of keeping the exposition clear, and the resulting bounds would be too small for applications, we do not pursue this here.
\end{remark}



\ \\

\begin{ack}
We would like to thank Professor Hajir for bringing this question to our attention, the attendees at the YMC 2010 conference for some conversations related to this work. Giuliana Davidoff, Joe Hughes and Adele Lopez point out an error in an earlier draft. The first named author was partially supported by NSF grant DMS0850577 and Williams College, the second named author by NSF grant DMS0970067 and the third named author by NSA grant H98230-05-1-0069 and NSF grant DMS0901506.
\end{ack}


\bibliographystyle{amsalpha}

\vfill
\hrule

\end{document}